%% file: adversarial-queues.tex
\documentclass[a4paper,11pt]{article}
\usepackage{verbatim}

\usepackage[style=alphabetic, url=false]{biblatex}
\usepackage{csquotes} % prevents a warning for using biblatex with babel

\usepackage[colorlinks=true,
                linkcolor=teal,
                citecolor=olive,
                destlabel=true,
                bookmarks=true]{hyperref}
\usepackage{xurl}
\hypersetup{breaklinks=true}

\usepackage{bookmark} 
\usepackage[margin=25mm,a4paper,includefoot]{geometry}
\usepackage[theorem/counter/prefix=section]{mac-bundle}
\usepackage{graphicx}

\usepackage{amsmath}
\usepackage{amsfonts}

\newfunction \End {End}
\newconstant \Set {Set}
\newconstant \B {\mathbb{B}}
\newfunction \Prob {\mathbb{P}}
\newfunction \blocks {\mathcal{B}}
\newcommand \cat {\mathop{|\!|}}

\newconstant \msg {\mathcal{M}}
\newconstant \states {\mathcal{X}}
\newconstant \addresses {\mathcal{A}}

\DeclareBibliographyDriver{eprint}{%
  \usebibmacro{bibindex}%
  \usebibmacro{begentry}%
  \usebibmacro{author/editor+others/translator+others}%
  \setunit{\printdelim{nametitledelim}}\newblock
  \usebibmacro{title}%
  \newunit
  \printlist{language}%
  \newunit\newblock
  \usebibmacro{byauthor}%
  \newunit\newblock
  \usebibmacro{byeditor+others}%
  \newunit\newblock
  \printfield{howpublished}%
  \newunit\newblock
  \printfield{type}%
  \newunit
  \printfield{version}%
  \newunit
  \printfield{note}%
  \newunit\newblock
  \usebibmacro{organization+location+date}%
  \newunit\newblock
  \usebibmacro{eprint}%
  \newunit\newblock
  \usebibmacro{addendum+pubstate}%
  \setunit{\bibpagerefpunct}\newblock
  \usebibmacro{pageref}%
  \newunit\newblock
  \iftoggle{bbx:related}
    {\usebibmacro{related:init}%
     \usebibmacro{related}}
    {}%
  \usebibmacro{finentry}}
\title{Adversarial blockchain queues and trading on a CFMM}
\author{Andrew W. Macpherson}
\begin{document}
\maketitle

\begin{abstract}
We describe a plausible probabilistic model for a blockchain queueing environment in which rational, profit-maximising schedulers impose adversarial disciplines on incoming messages containing a payload that encodes a state transition in a machine.
The model can be specialised to apply to chains with fixed or variable block times, traditional priority queue disciplines with `honest' schedulers, or adversarial public mempools.
We find conditions under which the model behaves as a bulk-service queue with priority discipline and derive practical expressions for the relative block and message number of a transaction.

We study this setup in the context of orders to a CFMM DEX where the execution price a user receives may be quite sensitive to its positioning in the chain --- in particular, to a string of transactions scheduled for prior execution which is not knowable at the time of order creation.
We derive statistical models for the price impact of this order flow both in the presence and absence of MEV extraction activity.
\end{abstract}

\section{Introduction}

In non-cooperative concurrent computation environments it is a fact of life that jobs must wait an unpredictable amount of time, during which an unpredictable number of other programs may execute, before being processed.
The amount of time that a job must wait depends on the algorithm used by the scheduler and the number and types of other jobs attempting to use the system.
Putting concrete figures to this is a task that traditionally falls within the domain of queueing theory, which aims to describe the distributions of quantities such as the \emph{number $N$ of jobs in the queue} and the \emph{wait time} $W$ given assumptions about the arrivals process, the service interval, and the queueing discipline \cite{erlang1909theory,gross2018fundamentals}.
The results may be used by system architects to provide statistical guarantees about server performance or client user experience under different workloads.

In typical queueing scenarios such as packets arriving at a network interface, the wait time may be considered the primary covariate of user experience.
However, in some highly time-sensitive environments the impact may be more complicated. 
For example, consider aiming a weapon in a multiplayer first-person shooter video game: a wait of a single frame between input and execution could mean the difference between a shot hitting its target and missing.
In algorithmic trading, meanwhile, mere microseconds can mean the difference between updating a limit order and having it fulfilled at an unfavourable (stale) price \cite{budish2015high}.

In blockchain execution environments, these issues are particularly pronounced for two reasons: 
\begin{enumerate}
  \item Jobs can be subjected to fairly arbitrary censoring and ordering disciplines by the block producer. 
  Reorderings can range over tens of seconds, break causality, and depend on message contents.
  \item Blockchain systems are typically used to record financial data, whence users may be particularly sensitive to execution outcomes.
\end{enumerate}

To model such environments, we must describe not only the \emph{number} of jobs in the queue, but also their \emph{contents}. Similarly, when modelling client experience associated to a particular message $o$, we must consider not only waiting \emph{time}, but also the number and contents of jobs executed while $o$ is pending, their impact on system state, and the dependence of user utility on final state.

A practical model would have value to blockchain or decentralised application architects, who could use it to provide statistical guarantees to users about executions. It could also be used to price derivatives such as insurance products and middlewares that provide fixed execution prices.

\subsection{A blockchain queueing environment}
In this work, we consider a queueing system with the following structure:
\begin{enumerate}

  \item
    Messages arrive into a message pool (or `mempool') $\bar\msg$ according to some point process.
    
    %While this process may be rather complicated and highly autocorrelated in general, there is some evidence that realistic dynamics can be derived even under a \emph{zero intelligence} hypothesis --- that message arrivals are a Poisson process and message contents are i.i.d.
    
  \item
    Messages are periodically batched by a scheduler, removed from the message pool, and appended to a log.
    %
    %It is assumed that this batching is a renewale intervals between batches are independent and identically distributed.
    
  \item
    The scheduler may intelligently \emph{inject messages} into the batch in order to maximise its own operator's utility.
    This leads to the notorious `dark forest' MEV extraction dynamics \cite{robinson2020ethereum}.
    
  \item
    Batches of messages are applied in order to a state machine to produce a state transition $\phi_0\mapsto\Phi_N$.
  
  \item
    Users are supposed to derive some utility depending on the final state $\Phi_N$ --- for example, because balances of some asset are recorded therein.
    
\end{enumerate}
If we pay attention only to the size $N=\#\bar\msg$ of the mempool, we have the classical situation of a bulk service queue --- typically $M/M^\beta/1$ (exponential block times) or $M/D^\beta/1$ (constant block times) in Kendall notation \cite{kendall1953stochastic}.
One analyses this by studying Kendall's \emph{embedded Markov chain} of the mempool size $N_i$ immediately before the $i$th batch.
In the cases of exponentially distributed or constant batch intervals, these methods yield concrete, computable formulae for the moments of the stationary distribution of $(N_i)_{i\in\N}$ \cite{bailey1954queueing}.
These formulae may be useful to blockchain node developers and operators in deciding how much memory to allocate for the mempool buffer, or to blockchain architects in designing block timing schemes.

On the other hand, calculations of client-side quantities --- quantities associated to a particular message in the queue rather than the queue as a whole --- depend on the \emph{queue discipline}, which is the function that pulls messages from $\bar\msg$ and orders them into batches.
This includes quantities relevant to `user experience' such as wait time, number of other jobs executed while waiting, and in the CFMM setting, price impact,
In traditional applications, it is assumed that the system designer has complete control over which discipline is used.
However, in permissionless, distributed blockchain systems, this fails for a number of reasons:
\begin{itemize}
  \item 
    Since the private view of a block producer is not usually independently verifiable, the system designer has little direct control of how it is decided which messages to include into a block.
  \item 
    Although it is possible to enforce rules about block \emph{ordering} (for example through verifiable sequencing rules \cite{ferreira2022credible}), the consensus rules of most major blockchain systems also leave this unspecified.
\end{itemize}
As a result, blockchain queueing disciplines can be `exotic;' for example, they may fail to be `work-conserving' in the sense that the scheduler always packs as many messages as possible from the mempool into the block, and they may depend on message contents.
Classical methods, which tend to focus on FIFO queues \cite{downton1955waiting}, need not apply.

Moreover, unlike in the classical theory, we are primarily interested in the client's \emph{utility} which is a function of the state of a machine (typically some kind of balance sheet) after executing his transaction $o$. 
Jobs scheduled in advance of $o$ impact the state in which it is executed, and hence their number and contents are covariates of user outcomes.

\subsection{Results}

In the first part of the paper, we introduce a queueing model that incorporates the contents of messages acting on a state machine and study information and incentive conditions under which it can be assumed to operate as a bulk service priority queue.
In \S\ref{priority-model}, we show how to derive the distributions of relative block number and position in the block in the limit of diffuse priorities (i.e.~the probability of any two messages having the same priority vanishes).
This limit is realistic in a setting where priorities are determined by fees, the space of plausible fees is very large, and clients apply some randomisation to fee settings (so there are no `focal points' where a fee collision is likely).
The unknowns in the formulas we obtain are readily computable either from real-world on-chain data or from standard $M/G/1$ models.

In \S\ref{markets}, we introduce the CFMM model and show how the priority queueing model can be combined with either an econophysics-style `zero-intelligence' (ZI) model for order flow or with real data of confirmed transactions to produce estimates of the variability of execution prices.
A robust queueing model allows us to do this without collecting mempool data, which may be inconvenient or expensive.

One might na\"ively hope that a realistic model may be obtained by a small perturbation of the ZI case.
However, as practitioners in Ethereum know well, perverse incentives can have non-perturbative effects.
In other words, you can't account for rational behaviour by simply making small adjustments to classical models. 
We exhibit this with a simple model where incentives force arbitrarily bad execution prices due to the possibility, hence necessity, of sandwich attacks.

\subsection{Related work}
\label{other-work}
\paragraph{Blockchain queueing theory}

Several groups of authors have applied queueing models to the case of blockchains to obtain formulae for expected transaction confirmation (i.e.~wait) times. 
In \cite{kasahara2019effect,kawasa2017transaction}, the authors consider a bulk service queue with general service times and priority discipline, a plausible model for `MEV-free' blockchains with priority fees. The formulas they obtain are expressed in terms of constants $\alpha_n$ which we do not know how to compute.
More explicit formulae are obtained in \cite{li2018blockchain,li2019markov} who assume a two stage service model that purports to incorporate network delays; however, their model assumes a FIFO discipline.
Meanwhile \cite{geissler2019discrete} discusses an intricate $GI/GI^N/1$ discrete-time model that allows transactions to have random sizes for the purposes of determining block capacity and validates it against simulations.
For a survey of these and related topics, see \cite{fan2020performance}.

Some authors consider block timing attacks in the context of bitcoin; for example, Markovian models have been applied to study selfish mining \cite{gobel2016bitcoin, javier2020further}.
Such multi-block strategies take us outside the regime where block production is a renewal process, and is outside the scope of our model.

\paragraph{MEV} The prevalence of MEV (maximal extractable value) activities on Ethereum was first brought to the attention of the academic community by \cite{daian2020flash}. See \cite{robinson2020ethereum} for an evocative introduction. 

The best known type of MEV occurs on constant function market maker decentralised exchanges (CFMM DEX), on whose basic dynamics there is a growing literature \cite{angeris2020improved, angeris2022constant, xu2022sok}.
In \cite{goyal2022finding} the authors discuss models of price movement under zero intelligence order flow.
Of particular relevance is \cite{kulkarni2022towards}, which attempts to define a general theory of MEV, including formulae for sandwich profitability --- compare \S\ref{sandwich}.
Meanwhile, \cite{0xblog2022measuring} provides an empirical study that gives some clue as to the actual prevalence of sandwiches, although on a higher abstraction layer than raw CFMMs.

Another thread of research aims to develop novel queueing disciplines adapted to trades on a CFMM, for example via verifiable sequencing rules \cite{ferreira2022credible} or first arbitrage auctions \cite{josojo2022mev,nikete2022towards}.

\subsection{Future directions}
\label{future-directions}
As far as I know, this work is the first attempt to define a probabilistic model of the blockchain mempool that incorporates both traditional queueing and incentive-based dynamics.
There are many directions in which future work could build on this model: empirical studies and more sophisticated modelling of the arrivals process; generalisations in which we drop the absolute time ordering to permit a distributed scheduler set; Markovian models in modelling the dependence of transaction arrivals on the previously observed mempool and state; more sophisticated information models in which the scheduler is blind to message payloads or other features.
Our model would be particularly interesting if it could be applied to quantitatively evaluate and compare modern approaches to `MEV-aware' queue design that have recently started to appear \cite{ferreira2022credible,josojo2022mev,nikete2022towards}.

\begin{comment}
\subsection{Glossary of notations}

We write $\# S$ for the cardinality of a set $S$, and $\ell(B)$ for the cardinality of an ordered set $B$.
\begin{itemize}
  \item $\msg$ --- space of messages
  \item $\bar\msg$ --- message pool (a.k.a.~mempool)
  \item $\widetilde\msg$ --- extended message pool (with injected messages)
  \item $\msg^a$ --- message arrivals process (a point process on $\msg\times\R$)
  \item $\varrho:\msg\rightarrow \states$ --- state space of the machine
  \item $(T_n)_{n\in\N}$ --- (i.i.d.) block intervals
  \item $\lambda$ --- message arrival rate
  \item $\mu$ --- block rate ($=1/\mathbb{E}(T_i)$)
  \item $o$ --- a message
  \item $W$ --- wait time
  \item $W_\#$ --- number of blocks to wait.
  \item $K$ --- position in block.
  \item $\widetilde{K}$ --- length of message prefix.
  \item CFMMs. 
    \begin{itemize}
      \item $C$ --- a CFMM (instance). 
      \item $f$ --- CFMM invariant. 
      \item $\phi_{C,k}$ --- reserves of $C$ at time step $k$. 
      \item $\lambda \defeq f(\phi_{C,k})$.
      \item $P$ --- price potential
    \end{itemize}
\end{itemize}
\end{comment}

\paragraph{Acknowledgements}

The author would like to thank Flashbots for funding and supporting this work through the FRP grants programme.\footnote{https://github.com/flashbots/mev-research} Thanks are also due to Quintus Kilbourn, Mohammed Al-Husari, and Aata Hokoridani, for useful conversations and comments.

\section{Model}

Our model is a bulk service queue with general (random) queueing discipline, random message injection, and a representation of queue elements, which are called \emph{messages}, on a state machine. That is, it comprises the following data:
\begin{itemize}

  \item
    A \emph{state machine} $(\mathcal{X},\varrho:\mathcal{M}\rightarrow\End(\mathcal{X}))$ and an \emph{initial state} $\phi_0\in\mathcal{X}$.
    
  \item
    A \emph{message arrivals process}, which is a marked point process $\msg^a\subset\msg\times\R$ of rate $\lambda>0$.
    
  \item
    An i.i.d.~sequence of random block intervals $\{T_n\}_{n\in\N}$ with expectation $\mu^{-1}$.
    
  \item
    A sequence of \emph{random blocks} $\{B_n\}_{n\in\N}$, that is, random ordered finite subsets of $\msg$ with length bounded by some constant $\beta\in\N$.
    
\end{itemize}
In the following sections, we describe various  other data that can be used to describe hypotheses on these processes.

Given an arrivals process and a blocks process, we can construct a \emph{message pool} $(\bar\msg_t)_{t\in\R}$ which is informally defined as follows:
\begin{enumerate}
\item Arriving messages are added to $\bar{\msg}$;
\item Messages are removed from $\bar\msg$ when included into a block.
\end{enumerate}
Blocks may also contain messages not in the pool at the time of block production; these messages are said to be \emph{injected}.
The union of the mempool with the set of injected messages gives the \emph{extended message pool} process $\widetilde\msg$.

From these data we can compute the \emph{state of the system after the $n$th block}, which is a random element of $\states$
\[
  \Phi_n \defeq B_n \cat\cdots\cat B_1 \cdot\phi_0,
\]
the string of messages acting through its representation $\varrho$ (here $\cat$ is the concatenation operator).

Suppose a message $o$ appears in $\bar\msg$ at time $0$, at which time the most recently observed state is $\phi_0$.
We will be interested quantities such as the following:
\begin{itemize}
  %\item The wait time $W$ until $o$ is included into a block $B_n$;
  \item The (relative) \emph{block number} $K'$ (starting at $0$ for the first block to appear after $t=0$) of the block into which $o$ in included;
  \item The \emph{position} $K''$ of $o$ in the block into which it is eventually included (indexed from $0$);
  \item The (relative) \emph{message number} $K = \beta K' + K''$ (where $\beta$ is the block size); that is, the number of messages applied to $\phi_0$ before $o$;
  \item The state $\Phi_{K'}\in\mathcal{X}$ immediately after the block containing $o$ is executed.
\end{itemize}
The first four quantities are discussed in this section, and the last in the case of a CFMM in \S\ref{markets}.

\begin{remark}[Wait time]

  In this paper we discuss neither the time value of money nor the time evolution of the user's private valuation of on-chain assets, either of which could factor in to a practical utility model via the \emph{time} $W$ until a given message $o$ is included into a block.
  In principle, the distribution of $W$ can be derived from understanding of $K'$ and $T$.

\end{remark}

\begin{remark}[Weaknesses of the model]

  Though quite general, the model sketched above and discussed in more detail below does not incorporate some features of real-life blockchain systems:
  \begin{itemize}

  \item The presentation of the arrivals and block processes assume an absolute time ordering. This is fine with a single block producer (or more generally with all block producers in the same physical location) but unrealistic in the presence of network delays.
  
  \item The model assumes that apart from messages injected \emph{by the block producer}, arrivals are a Poisson process. In reality, other agents may watch the mempool and submit transactions depending on previous arrivals.
  Empirically, these types of transactions actually make up a significant proportion of message volume on DEXes \cite{alexthuniswap}.
  
  \end{itemize}
  
\end{remark}

\paragraph{Notational conventions}

Generally, capital letters denote random variables and their lower case equivalence specific realisations. The cardinality of a set $S$ (resp.~ordered set $B$) is denoted $\# S$ (resp.~$\ell(B)$).
The timescale is taken to be $\R$, though it could be replaced with $(0,\infty)$ without much changing.
We write $X\stackrel{\$}{\in}S$ to mean $X$ is random variable valued in the (measurable) set $S$.

Now let us consider in more detail some of the less standard components of this model.

\subsection{Message pool}

We will need to make use of random countable subsets of a given measurable set $\msg$ or $\msg\times\R$, a.k.a.~simple point processes.
These can be made sense of using random counting measures.
A random subset of a random set $\msg'\stackrel{\$}{\subseteq}\msg$ is a random subset $\msg''\stackrel{\$}{\subseteq}\msg$ such that $\msg''\subseteq\msg''$ almost surely.

When the ambient space is $\msg\times\R$ we may use the following simple-minded approach: a random subset consists of a sequence of tuples $(U_i,M_i)_{i\in\N}$ where $U_i>0$ (called a \emph{holding time}) and $M_i\stackrel{\$}{\in} \msg$.
Such a random subset is \emph{marked Poisson} if $U_i$ are i.i.d.~exponentially distributed random variables.
This hypothesis is ubiquitous in the queueing theory literature.

\paragraph{Messages}
The \emph{message arrivals process} is assumed to be a marked Poisson process $\msg^a\stackrel{\$}{\subset} \msg\times\R$ with rate $\lambda$.
If $I\subseteq\R$ is an interval, write 
\[
  \msg^a_I\defeq \msg^a\cap (\msg\times I)
\]
for the set of messages arriving during $I$. The cardinality of this set has the Poisson distribution with rate $\lambda\cdot|I|$.

The classical theory gives us plenty of tools to make hypotheses and derive conclusions about the arrival holding times $U_i$. This leaves the question of the distributions of the $\msg$-valued discrete time process $\{M_i\}_{i\in\N}$, of which we defer further discussion to \S\ref{markets}.

\paragraph{Blocks}

\begin{definition}[Blocks]
  
  A \emph{block} on a set $S$ is a finite, totally ordered set $(B,\leq)$ together with an injective map $B\rightarrow S$. The number of elements in a block is called its \emph{length}, written $\ell(B)$. The set of blocks on a set $S$ is denoted $\blocks(S)$.

\end{definition}  

The set of blocks on a set $S$ can be decomposed as a disjoint union
\[
  \mathcal{B}(S) = \coprod_{J\subseteq S}\{\text{orderings of }J\} \simeq \coprod_{J\subseteq S}\Sigma_J
\]
of the sets of permutations of all subsets of $S$ (where for the second identification we use some ordering of $S$). This decomposition is useful for constructing distributions on the set of blocks in two stages:
\begin{enumerate}
  \item An \emph{inclusion} distribution, a distribution on the set of finite subsets of $S$.
  \item Conditioned on a given set $J\subseteq S$ of inclusions, an \emph{ordering} distribution on $\Sigma_J$.
\end{enumerate}
If $T\stackrel{\$}{\subseteq} S$ is a random subset, then a \emph{block on $T$} is a random block $B$ on $S$ whose underlying random set is a subset of $S$ (almost surely).

\begin{definition}[Priority]
\label{priority-discipline}

  Suppose $\msg$ comes equipped with a \emph{priority} function $p:\msg\rightarrow\R$.
  A (random) block $B=(M_1,\ldots,M_K)$ on $\bar{\msg}\stackrel{\$}{\subseteq}\msg$ is said to obey
  \begin{itemize}
    \item \emph{priority inclusion} if $\min\{p(M_i)|i\leq K\} \geq \sup_{\bar{\msg}\setminus B}(p) $ almost surely.
    \item \emph{priority order} if $\Prob(p(M_i) < p(M_j)) = 0$ for all $i<j$.
  \end{itemize}
  Note that priority order can be determined without reference to $\bar{\msg}$, but priority inclusion cannot.
  Thus only the former can be enforced as a verifiable sequencing rule \cite{ferreira2022credible}.

\end{definition}
  
\paragraph{Message pool}
The message pool process is a locally constant (cadl\`ag) family of subsets $(\bar{\msg}_t)_{t\in \R}$ of $(\msg^a_{\leq t})_{t\in\R}$ satisfying the following conditions:
\begin{enumerate}
  \item $\bar{\msg}_t = \bar{\msg}_{T_n}\sqcup \msg^a_{[T_n,t]}$ for $t\in [T_n,T_{n+1})$. 
  \item $\lim_{t\downarrow T_n}\bar{\msg}_t = \lim_{t\uparrow T_n}\bar{\msg}_t \setminus (B_n\cap \lim_{t\uparrow T_n}\bar{\msg}_t)$.
\end{enumerate}
These conditions formalise the assumption that messages are removed from the pool when they are included into a block.
The state of the pool just after this operation is $\bar\msg_{T_n}$.
Note, however, that we have not assumed that the block consists entirely of messages from the pool.
We return to this feature in \S\ref{message-injection}.

The message pool process can be recovered from $\msg^a$ and the block arrivals process $(T_n,B_n)$ by the formula
\[
  \bar{\msg}_t = \msg^a_{\leq t} \setminus \left ( \bigcup_{\{n|T_n \leq t\}} B_n \cap \msg^a_{\leq t} \right).
\]
We write $\bar{\msg}_n \defeq \lim_{t\uparrow T_n}\bar{\msg}_{T_n}$ for the set of messages in the pool just before the $n$th block.
Following Kendall \cite{kendall1953stochastic}, we then call it the \emph{embedded process} of $(\bar\msg_t)_{t\in\R}$.
Note that it need not be Markovian.

\subsection{Message injection}
\label{message-injection}

At each block time $T_n$, the block is produced by some function called the \emph{scheduler}.
As well as drawing messages from the message pool at time $T_n$, the scheduler may introduce their own messages.
In the terminology introduced by \cite{wolff1970work}, our blockchain queue may fail to be \emph{work-conserving}, i.e.~the scheduler always includes as many messages as possible from the mempool into the block. Injected messages take up additional `work' (blockspace) not entailed by the arrivals process alone.

\begin{remark}

  In applications it is also pertinent, but beyond the scope of this paper, to consider the scheduler as a random algorithm (rather than merely considering its output as a random variable) so that one can make hypotheses on its complexity. Compare \cite[\S3]{ferreira2022credible}.
  
\end{remark}

The message pool and blockchain together also define another discrete time process
\begin{equation}
  \widetilde{\msg}_n \defeq \bar{\msg}_n \cup B_n
\end{equation}
which we call the \emph{extended} message pool process.
Elements of the difference $\widetilde\msg_n\setminus\bar\msg_n$ are \emph{injected} messages.
Injected messages are qualitatively different from arriving messages because by definition they always make it into the block.
Nonetheless, we might hope in some cases to derive the set of injected messages from some queue process to which we can apply classical models.

If $N_{\mathrm{all},n}=\#\widetilde\msg_n$ were Poisson distributed with rate $\lambda_\mathrm{all}T_n$ where $T_n$ is the $n$th block interval, then it would be distributed as the embedded process of an $M/GI^\beta/1$ queue and hence the server-side quantities ($\#\widetilde\msg$, busy periods, clearing times) could be computed with an $M/GI^\beta/1$ model.
Because of the bound on block size, this cannot literally be the case, but for sufficiently small $\lambda/\mu\beta$ it may sometimes be a reasonable approximation:

\begin{example}[Na\"ive scheduler]
\label{naive-scheduler}

  Suppose that block times are constant and the scheduler of block $B_n=(M_1,\ldots M_{\ell(B_n)})$ is \emph{na\"ive} in that the predicates $M_k\in \bar\msg$ are i.i.d.~(Bernoulli) random variables.
  That is, each message position in the block has an equal probability $p$ of holding an injected messages, and these events are all independent.
  Then the number of injected messages is $\sim\mathrm{Binom}(\beta,p)$.
  If $\beta$ is large, this is approximated by a Poisson distribution with rate $\beta p$.
  Hence the server-side quantities can be computed in terms of an $M/D^\beta/1$ model \cite{bailey1954queueing}.
  
  A na\"ive scheduler operator might be someone who runs the reference implementation of the blockchain node software and interacts with it as an ordinary client.
  Alternatively, it may be a scheduler operating in an environment where it cannot choose which transactions are included into blocks --- for example, because of some cryptographically enforced randomisation scheme.
  
\end{example}

It may also be pertinent to restrict the set of messages the scheduler is allowed to inject to some subset $\msg_n^i\subseteq\msg$, say, of messages with authorisation associated to some element of a fixed set of keys.
More generally, since the operator constructing a given block may be selected by a random election, we may ask $\msg_n^i$ to be a random subset.

\begin{example}

  Assume the accounts + token balances model (ex.~\ref{token-balance-model}) and that all transactions are value transfers (\ref{value-transfer}). Suppose there is a set $\mathcal{S}$ of scheduler operators and that to each element $s\in \mathcal{S}$ there is associated a set $\mathcal{A}_s\subseteq\mathcal{A}$ of addresses.
  Then an $\mathcal{S}$-valued random variable $S$ yields a random subset $\msg_S\subseteq \msg$ consisting of the value transfers with sending address in $\mathcal{A}_S$.
  
\end{example}

\subsection{State} 
A \emph{state machine} $(\mathcal{X},\varrho:\mathcal{M}\rightarrow\End(\mathcal{X}))$ consists of the data of a set (or more structured object such as a vector space) $\mathcal{X}$ and a representation on $\mathcal{X}$ by some set $\msg$. 
The elements of $\mathcal{M}$ are called \emph{messages} and $\varrho(m)$ is called the \emph{transaction payload} of $m$. 

We suppose given a $\sigma$-algebra on $\msg$ and a compatible $\sigma$-algebra on a subset of $\End(\mathcal{X})$ containing the image of $\varrho$.
This allows us to talk about random messages and random transactions.

\begin{example}[Key-value store]

  State has the structure of a \emph{key-value store} so that we have a decomposition
  \[
    \mathcal{X} = \prod_{k\in K}\mathcal{X}_k
  \]
  (where $K$ is the set of keys and $\mathcal{X}_k$ the set of possible values associated to a given key $k\in K$).
  
  This structure allows us to formulate conditions like ``the function $v:\mathcal{X}\rightarrow \Omega$ depends only on fields $k\in K_0\subseteq K$,'' meaning that $v$ factors through the projection $\mathcal{X}\rightarrow\prod_{k\in K_0}\mathcal{X}_k$.
  
\end{example}

In the rest of the paper, we will in fact assume that $\mathcal{X}$ has the structure of a key-value store. We will make hypotheses on the structure of certain fields $\mathcal{X}\rightarrow\mathcal{X}_i$, tacitly assuming we are given a product decomposition $\mathcal{X}=\mathcal{X}_i \times \mathcal{X}_{-i}$. This allows us to specify a transaction on $\mathcal{X}$ by specifying one on $\mathcal{X}_i$ and stipulating that other fields are left untouched. 

%This does not lose any real generality, either practically (for why?) or theoretically (avoiding the use of product decompositions complicates the language of certain assumptions because we must talk about choosing a lift rather than going with the canonical one).

\begin{example}[Accounts model]
\label{accounts-model}
  
  $\mathcal{X} = \mathcal{X}_\ell^\mathcal{A} \times E$ where $\mathcal A$ is a type of 160-bit addresses, $\mathcal{X}_\ell$ is a "local state" type, and $E$ consists of ``environment information'' such as block number.
  
\end{example}

\begin{example}[Token-balance mapping]
\label{token-balance-model}
  
  A toy model for local state is $\mathcal{X}_\ell=[0,\infty)^T$ where $T$ is a fixed set of token types and an element of $L$ is interpreted as a vector of balances.
  State transitions are \emph{conditional balanced translations} parametrised by elements $(r_a)_{a\in\mathcal{A}}$ of $\mathcal{X}_\ell^\mathcal{A}\otimes \R$ with finite support such that $\sum_{i\in A}a_i=0$. 
  The state transition is defined by
  \[
    (x_a)_{a\in\mathcal{A}} \mapsto 
    \left\{\begin{array}{ll} 
      (x_a+r_a)_{a\in\mathcal{A}} & x_a+r_a \geq 0 \quad \forall a\in\mathcal{A} \\ 
      (x_a)_{a\in\mathcal{A}} & \text{otherwise} 
    \end{array}\right. 
  \]
  This model has no mints, burns, or nonlinear derivatives but is still enough to model a swap in an ownerless pool.
  
\end{example}

\begin{example}[Value transfer]
\label{value-transfer}

  A \emph{value transfer} transaction in the token-balance mapping model is a translation parametrised as above by $(r_a)_{a\in\mathcal{A}}$ with support in two addresses $\{S,R\}$, the \emph{sender} and \emph{receiver}, and a single token %t\in T$.
  Value transfers can be described in terms of these three data and the amount to be transferred.
  
\end{example}

\begin{remark}
  
  Intuitively, elements of the set $\mathcal{A}$ could represent individual human agents but also aggregations such as companies or markets. Indeed, to any subset $I\subseteq\mathcal{A}$ we can associate a state vector 
  \[
    \phi_I\defeq \sum_{i\in I}\phi_i \in \mathcal{X}_I\defeq [0,\infty)^{\mathcal{T}}
  \]
  representing the total assets of the collection of agents in $I$.
  
  Concretely, an element of $\mathcal{A}$ could represent a single Ethereum address (which itself could identify an individually held EOA, a multisig wallet, or an AMM, for example), or a set of addresses (representing many addresses held by a single offchain entity or even a set of entities).
  
\end{remark}

\begin{example}[Metadata]

  The representation $\msg \rightarrow \End(\mathcal{X})$ need not be injective. For example, we might have $\msg = \msg_m\times\End(\mathcal{X})$ where $\msg_m$ is a `metadata' type encoding information such as authorisation (signature), timestamp, and fee limits.

\end{example}

\subsection{Preferences}
Suppose that the most recently observed state is $\phi_0\in\mathcal{X}$ and that the next block is being produced by a rational agent with an utility function $U:\mathcal{X}\rightarrow\R$. 
We define the utility of a block $b$ as $U(b)\defeq U(b\phi_0) - U(\phi_0)$.
Then if $b\in\blocks(\msg)$ is a realisation which is \emph{dominated} by some other $b'$ --- that is, $U(b) < U(b')$ --- we should expect that the scheduler never outputs $b$.

\begin{definition}[Rationality]
\label{rational}

  Say that a random block $B$ is \emph{rational} with respect to $U:\states\rightarrow\R$ if any $\mathbb{P}(B=b)=0$ for any $b\in\blocks(\msg)$ which is dominated by some other block.

\end{definition}

Designing the state machine so that node operator's utility is maximised by engaging in desirable behaviour is a basic tenet of permissionless blockchain design.
We now discuss some scenarios under which rationality can be leveraged to produce desirable behaviour.

\begin{example}[When are blocks full?]
\label{positive-utility}

  A message $m$ has \emph{positive utility} given a block $b\stackrel{\$}{\subseteq}\widetilde\msg$ with $\ell(b)<\beta$ if 
  \[
    U(b) < U(b' \cat m\cat b'')
  \]
  for some blocks $b',b''$ such that $b = b'\sqcup b''$ as unordered sets.
  Say $m$ has positive utility (unconditionally) if it has positive utility given any non-full block.
  
  Rational blocks never leave both space in the block and positive utility messages in the extended mempool, though they may leave positive utility messages in the (unextended) mempool if it is more valuable to inject messages.
  If $B$ is a rational block on $\widetilde\msg$ and all elements of $\widetilde\msg$ have positive utility,
  \[
    \Prob(\ell(B)=\beta | \#\widetilde\msg \geq \beta ) = 1.
  \]
  
\end{example}

\begin{example}[Priority inclusion]
\label{priority-inclusion-rational}

  Suppose given a priority function $p:\msg\rightarrow\R$ (Def.~\ref{priority-discipline}) such that for all $m,m'\in\msg$ and blocks $B_\pm\subseteq\msg$,
  \[
    p(m) > p(m') \quad \Rightarrow \quad U(B_-\cat m\cat B_+) > U(B_-\cat m'\cat B_*).
  \]
  We call $p$ \emph{incentive-compatible} w.r.t.~$U$.
  In this case, a rational block always obeys $p$-priority inclusion.

\end{example}

\begin{remark}[Empty blocks and utility]
  
  It may be tempting to assume that all messages that carry a fee have positive utility.
  However, as witnessed by the phenomenon of empty blocks on Bitcoin and Ethereum \cite{gauthier2016why,silva2020impact}, this is not always realistic.
  On programmable platforms like Ethereum, there may be many incentives to censor a message, even if there are no other messages in the mempool to take its place.

\end{remark}

\subsection{Priority model}
\label{priority-model}

In this section, we specialise to the case where $\msg$ comes with an incentive-compatible priority function $p$ as in Def.~\ref{priority-discipline}.
As we have seen, rational blocks are either full or leave no unconditionally positive utility messages in the message pool, and they obey priority inclusion.

Let $o\in\msg$ be a message arriving the the mempool. We discuss the block number $K'$ of the block into which $o$ is eventually included (starting from block number $0$ for the first block after $o$ is issued) and position $K''$ of $o$ in that block (indexed from $0$) in the context of a chain of rational blocks.
Call all messages with higher (resp.~lower) priority than $o$ \emph{high} (resp.~\emph{low}) \emph{priority}, and let $(N_{i})_{i\in\N}$ be the number of high priority messages in the message pool just before the $i$th block.

Suppose also:
\begin{enumerate}
  \item 
    All blocks are rational (Definition \ref{rational})
    
  \item
    Priorities are incentive-compatible (Example \ref{priority-inclusion-rational}).
    
  \item 
    All injected messages have high priority.
    
  \item 
    Both $o$ and any high priority message has unconditionally positive utility (Example \ref{positive-utility}).
    
  \item 
    No other message has exactly the same priority as $o$.
    
    In practice this may be reasonable in applications where fees are defined to a precision of billionths of a dollar and some fee randomisation occurs.
    
  \item
    The numbers $\{S_i\}_{i\in\N}$ of high priority arrivals in each block interval are i.i.d.~random variables. 
     
    This is the case, for example, if high priority arrivals obey a Poisson process.
    
  \item
    The high priority queue counting process $(N_{n})_{n\in\N}$ is strongly stationary.
    
\end{enumerate}

\begin{remark}[Independence of high priority arrivals process from arrival time of $o$]

  The hypothesis on the number of arrivals per block interval implies that the number of arrivals in a given interval is not related to relative position to $o$ (since that is how we index blocks).
  Even accepting that this means the high priority process does not somehow `react' to the appearance of $o$, this may be unrealistic in the case of variable block times because $o$ is more likely to appear in a longer block interval where, at least under standard $GI/GI^\beta/1$ assumptions, we would also expect to see more arrivals.
  Taking proper account of this phenomenon would necessitate special treatment of the first block interval.
  
\end{remark}

By rationality and unconditional positive utility, the high priority queue is work-conserving.
That is, the number of messages removed from the high priority queue at step $n$ is $\min(\beta,N_{n})$, whence the queue length process $(N_{n})_{n\in\N}$ is a random walk with step size $S-\beta$ (where $S$ has the distribution of $S_i$ for any $i$) and barrier at $0$.
In particular, it is homogeneous Markovian.
Let us write
\begin{equation}
  P^I_J \defeq \Prob(N_{n} \in J \mid N_{n-1}\in I) 
\end{equation}
for the Markovian transition probability from event $I\subseteq\N$ to $J\subseteq\N$.

\begin{lemma}
\label{markov-transition-ratio}

  Let $I,J\subseteq\N$ with $I\subseteq[\beta,\infty)$. We have the identity
  \[
    P^I_J = \frac{\sum_{i\in I}\Prob(N=i)\cdot \Prob(i+S-\beta\in J)}
    { \Prob(N\in I) }.
  \]
  
\end{lemma}
\begin{proof}

  The hypothesis on $I$ implies that $N_n\in I\Rightarrow N_{n+1}=N_n+S-\beta$, so we may treat $N$ as a random walk (without worrying about the barrier).\qedhere

\end{proof}

\begin{proposition}[Distribution of message number]
\label{message-number-distribution}

  Under the hypotheses 1--7, we have
  \begin{align}
    \Prob(K'=0) &= \Prob(N< \beta) \\
    \Prob(K'=k) &= \Prob(N<\beta)\cdot \left(1-P^{\geq\beta}_{<\beta}\right)^{k-1}\quad k>0.
  \end{align}
  In particular, the restricted random variable $K'-1$ given ${K'>0}$ is geometrically distributed with failure probability $P^{\geq\beta}_{\geq\beta}$.
  
  If, moreover, blocks obey priority ordering, then the distribution of $K''$ given $K'=0$ (in which case $K=K''$) is as follows:
  \begin{align}
    \Prob(K''|_{K'=0}= k) &= \frac{\Prob(N=k)}{\Prob(N<\beta)}  \label{K''-dist-0} \\
   \Rightarrow \qquad \Prob(K=k) &= \Prob(N=k) \quad k=0,\ldots,\beta-1.
  \end{align}
  Meanwhile, $K'$ and $K''$ are conditionally independent given $K'>0$, and the distribution of $K''$ is given by:
  \begin{align}
    \Prob(K''|_{K'>0}= k) &= \frac{P^{\geq\beta}_k}{P^{\geq\beta}_{<\beta}}  \\
    &= \frac{\sum_{i=0}^{k} \Prob(N=i+\beta)\cdot \Prob(S=k-i)} 
      {\sum_{i=0}^{\beta-1} \Prob(N=i+\beta)\cdot \Prob(S<\beta-i)}.
  \end{align}
\end{proposition}
\begin{proof}

  Since rational blocks obey priority inclusion, for the purposes of computing the block number of $o$, we may discard low priority messages.
  
  Regardless of arrival time, $o$ gets into a given block $B_i$ iff it does not get into any previous block $B_j$, $j=0,\ldots,i-1$ and there is space left over after all high priority transactions are packed into $B_i$. 
  That is,
  \begin{align*}
    \Prob(o\in B_i) &= \Prob(o\in B_i\mid o\not\in B_{j},\;j=0,\ldots,i-1) \\
    &= \Prob(N_i < \beta \mid N_j\geq\beta\;j=0,\ldots,i-1) \\
    &= \Prob(N_i < \beta \mid N_{i-1}\geq\beta) \qquad \text{(Markov property)}
  \end{align*}
  for each positive integer $i$, and
  \[
    \Prob(o\in B_0) = \Prob(N<\beta)
  \]
  since $o\not\in B_{-1}$ by construction.
  This proves the first part.
  
  Now suppose blocks obey priority ordering. 
  Then $K''$ is the number of high priority messages in $B$ given $o\in B$, which satisfies
  \begin{align*}
    \Prob(K''=k\mid K'=0) &= \Prob(N=k|N < \beta) = \frac{\Prob(N=k)}{\Prob(N<\beta)}  \\
    \Prob(K''=k\mid K'>0) &= \Prob(N_{i}=k\mid N_{i}<\beta, N_{i-1}\geq\beta) \\
    &= \frac{\Prob(N_{i}=k\mid  N_{i-1}\geq\beta)}
    {\Prob(N_{i}<\beta \mid N_{i-1}\geq\beta)}
  \end{align*}
  for $k=0,\ldots,\beta-1$.
  The last equality follows from Lemma \ref{markov-transition-ratio}.
  \qedhere

\end{proof}

\begin{corollary}
\label{message-number-finiteness}

  The distribution of $K$ is completely determined by either of the following finite lists of numbers:
  \begin{itemize}
    \item $\Prob(N = n)$ for $n=0,\ldots,2\beta-1$ and $\Prob(S=k)$ for $k=0,\ldots,\beta-1$.
    \item $\Prob(N = n)$ for $n=0,\ldots,\beta-1$ and $P^{\geq\beta}_k$ for $k=0,\ldots,\beta-1$.
  \end{itemize}
  
\end{corollary}

Corollary \ref{message-number-finiteness} tells us we can build a model of $K$ using information that is likely to be available in practice.
Assume the high priority queue is stable in that $\lambda_\mathrm{high} < \mu\beta$.
Given a priority bound, the numbers $\Prob(N=n)$ can reasonably be estimated empirically for $n=0,\ldots,\beta-1$: by the ergodic theorem, it is the proportion of blocks containing exactly $n$ high-priority messages.
The numbers $P^{\geq\beta}_k$ for $k=0,\ldots,\beta-1$, and hence $P^{\geq\beta}_{<\beta}= \sum_{k=0}^{\beta-1}P^{\geq\beta}_k$, may be similarly measured as the proportion of blocks containing exactly $k$ high priority messages following a block completely filled with high priority transactions.

Alternatively, we may take a theoretical approach to deriving these numbers. If the high priority queue has Poisson arrivals, the distribution of $N$ can be analysed as the stationary queue size of an $M/GI^\beta/1$ work-conserving bulk service queue.
When block times are Erlangian (e.g.~constant or exponential), the moments of this distribution were deduced in \cite{bailey1954queueing}.
In particular, we can obtain exact expressions for the distributions of $N$ and $S$.
    
\begin{example}[Exponential block time]
\label{exponential-block-time}

  If high priority arrivals are Poisson and (as in Bitcoin \cite{kasahara2019effect}) block times are exponentially distributed then $N$ ends up geometrically distributed with failure probability $p$, where $p$ is the unique root in $(0,1)$ of the polynomial
  \[
    \chi(p) = \mu p^{\beta+1} - (\lambda + \mu)p + \lambda,
  \]
  where we abbreviate $\lambda\defeq\lambda_\mathrm{high}$ (cf.~\cite[\S3.2]{gross2018fundamentals}).
  Note that this quantity only depends on $\beta$ and $\lambda/\mu$, the expected number of high priority arrivals per block.
  In particular, $\Prob(o\in B_0) = 1-p^{\beta}$.
  When $\beta$ is large (and $p<1$), the leading order term in $\chi(p)$ is small and so we deduce $p\approx \frac{\lambda}{\lambda+\mu}$, the proportion of events made up of arrivals.
  The approximation is very good for realistic $\beta$ and and moderate $\lambda/\mu$, but the error gets amplified by the exponent $\beta$ in the formula for $\Prob(o\in B_0)$.
  \begin{center}
    \includegraphics[height=30ex]{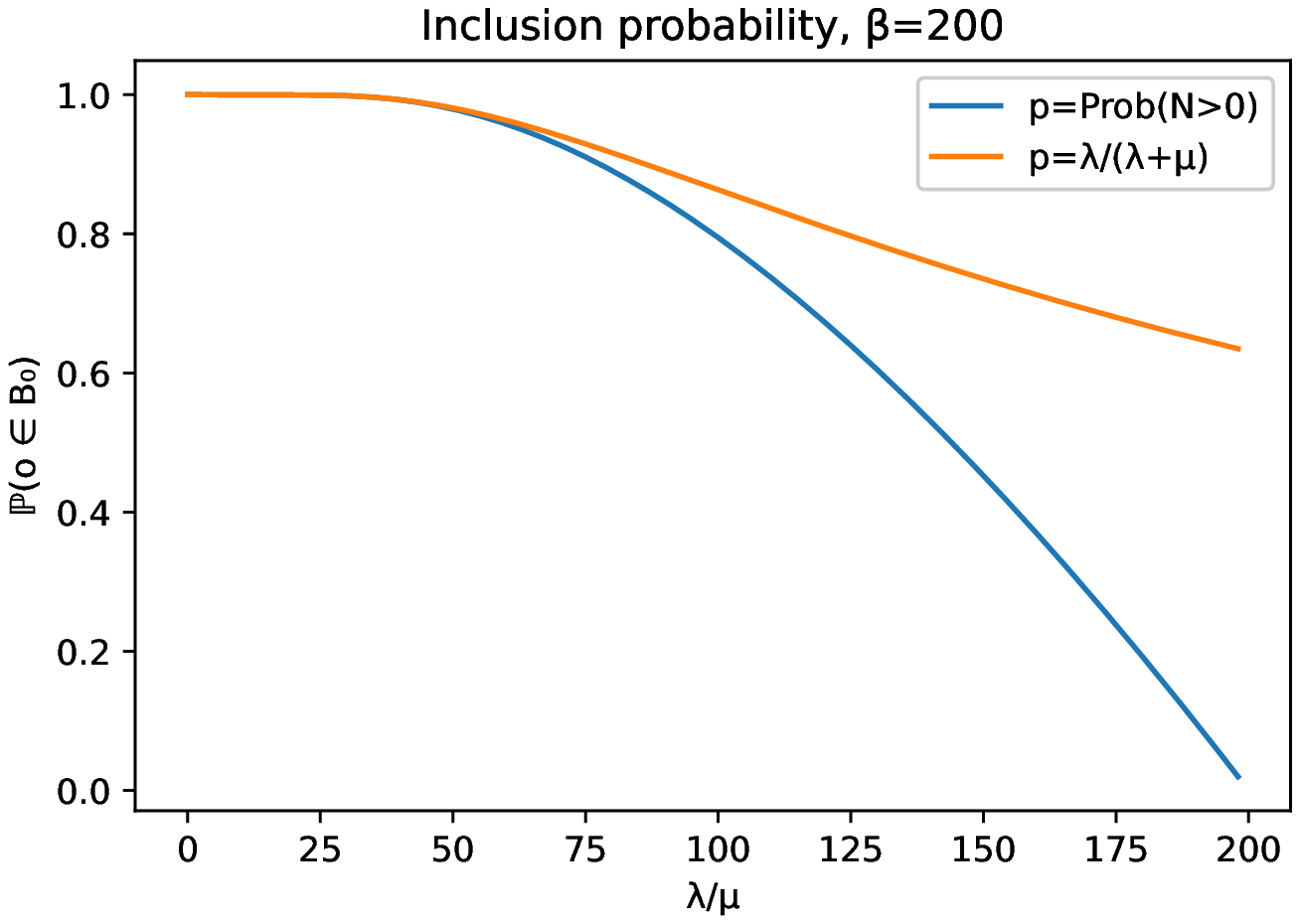}
    \includegraphics[height=30ex]{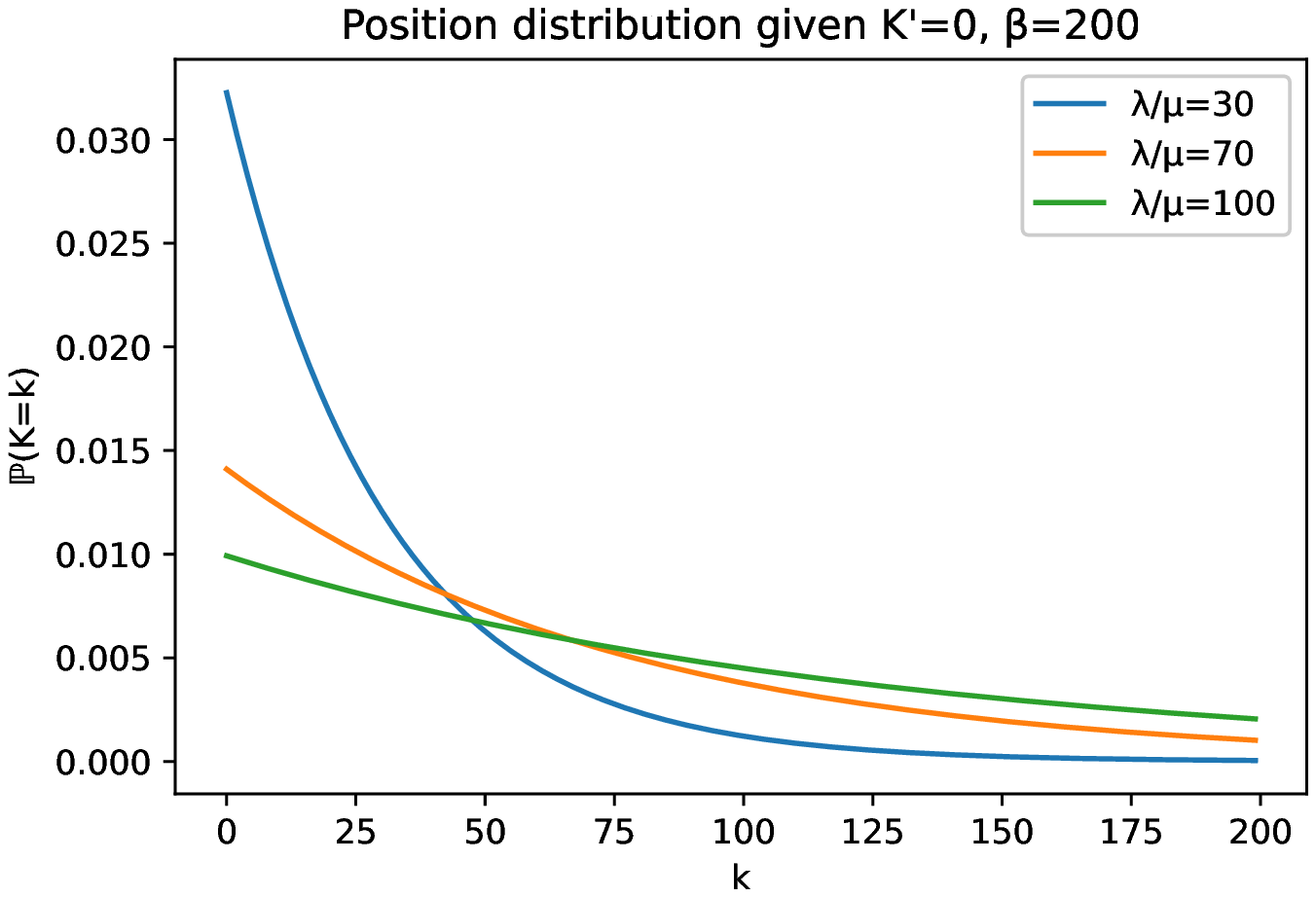}
  \end{center}    
  For the position in the block given inclusion in $B_0$, we obtain
  \[
    \Prob(K''=k \mid K'=0 ) = p^k / (1-p^\beta)
  \]
  for $k=0,\ldots,\beta-1$. In particular, positions earlier in the block are more likely.
  
  When $K'>0$, we have to incorporate the distribution of number of high priority arrivals per block, which under present hypotheses is Poisson distributed with rate $T\lambda$ where $T\sim\mathrm{Exp}(\mu)$.
  This yields a geometric distribution with failure probability $q=\mu/(\lambda+\mu)$ ($\approx 1-p$).
  From this we can derive tractable formulas for the transition probabilities, and hence the distributions of block number and block position given $K'>0$. 
  For example, using the fact that $\Prob(N=\beta+i|N\geq\beta)= \Prob(N=i)$ because $N$ is geometrically distributed, we obtain
  \begin{align}
    P^{\geq\beta}_k &= \sum_{i=0}^{k}\Prob(N=i)\cdot \Prob(S=k-i) = (1-p)(1-q)\sum_{i=0}^{k}p^iq^{k-i} \\
    P^{\geq\beta}_{<\beta} &= \sum_{i=0}^{\beta-1}\Prob(N=i)\cdot \Prob(S<\beta-i) = (1-p)\sum_{i=0}^{\beta-1}p^i(1-q^{\beta-i}).
  \end{align}
  
\end{example}

\begin{remark}

  The Poisson arrivals model may break down substantially for injected or other very high priority messages in the presence of recurring top-of-block opportunities (MEV).
  In this case, a special model is needed to handle low $K''$.

\end{remark}

\section{Markets}
\label{markets}

We now turn to discussing the impact of higher priority messages on the execution price of an order on a constant function market maker (CFMM) DEX.

\paragraph{Trades}

Assume the addresses and token balances model. Suppose fixed an ordered pair of tokens $A,B$ and two addresses $P,C$ (the \emph{originator} and the \emph{counterparty}).
A \emph{trade} $T$ between $P$ and $C$ consists of the data of a vector $T = (t_A,t_B)\in \R^2$ such that $t_A$ and $t_B$ have opposite sign.
It defines a transaction $\tau_T\curvearrowright \states$ that acts as translation by 
\[ \tau(t_A,t_B) = (t_A,-t_A,t_B,-t_B) \]
on $\mathcal{X}_{A,P}\times\mathcal{X}_{A,C}\times\mathcal{X}_{B,P}\times\mathcal{X}_{B,C}$ (and as the identity on all other fields) provided $P$ and $C$ have sufficient balance, or as the identity otherwise.

If the trading parties all have sufficient balance, a string of trades acts as translation by the sum of the corresponding vectors.
In particular, the transactions all commute.
If balances run low, some of the vectors may have to be omitted from the sum, depending on ordering, and commutativity is lost.

\paragraph{Constant function market maker} Constant function market makers (CFMMs) are a class of on-chain applications that provide passive liquidity by accepting trades that preserve a function of the reserves (before fees).
In particular, execution prices are defined in terms of this function and the starting reserves.
CFMMs are perhaps the best-studied and most interacted with entities in DeFi \cite{angeris2020improved, angeris2022constant, xu2022sok,kulkarni2022towards}. We briefly review the basic definitions and notation we will use:

\begin{itemize}
  \item 
    A \emph{CFMM} comprises the data of a twice continuously differentiable function 
    \[
       f:(0,\infty)^2 \rightarrow \R
    \]
    whose partial derivatives with respect to both standard coordinates are strictly positive.
    We will call this function the \emph{invariant}.\footnote{It is also often called the \emph{trading function}, but as this meant something different in some of the early papers on the subject I avoid it as a potential source of ambiguity.}

    Level sets of the conserved quantity $f$ are denoted as
    \[
      W_\Lambda \defeq \{(x,y)\in (0,\infty)^2|f(x,y)=\Lambda\}.
    \]
    
  \item
    We label the axes of $(0,\infty)^2$ by symbols $A$, $B$ which in concrete instantiations will be the names of two tokens. 
    We call $A$ the num\'eraire and express prices in terms of $\mathtt{A\_amount/B\_amount}$. 
    The two coordinate functions are written $r_A$ and $r_B$.
  
  \item
    Positivity of the partial derivatives of $f$ entails that the restrictions to $W_\Lambda$ of $r_A,r_B$ are open immersions with opposite orientation. For simplicity, we assume that these are surjective for $\Lambda$ in the range of $f$.
    Precomposing one of these with the inverse of the other yields an (orientation-reversing) \emph{cost potential}
    \[
      P_\Lambda = r_B\circ ( r_A|_{W_\Lambda})^{-1}:(0,\infty)\rightarrow (0,\infty)
    \]
    which yields the reserves of asset $B$ that balances a given amount of asset $A$ at level $\Lambda$.%
    We will suppress the dependence of $P$ on $\Lambda$ in cases where the expression as a whole does not depend on $\Lambda$ (but beware that this practice is potentially confusing in the presence of trading fees which increase $\Lambda$).

  \item
    The amount of token $B$ received in exchange for $a$ units of token $A$ starting from state $(\phi_A,\phi_B) = (\phi_A,P(\phi_A))$ --- pushing the state to $(\phi_A+a,P(\phi_A+a))$ --- is
    \[
      P(\phi_A) - P(\phi_A+a)
    \]
    units of asset $A$. This value is also called the \emph{exchange function} $F(1)$ (given reserves $\phi$) \cite[\S4.1]{angeris2022constant}.
    From this we can compute the (2-dimensional) \emph{payoff} of the swap for the originator as
    \begin{equation} \label{payoff}
      (-a, P(\phi_A) - P(\phi_A+a)) \in \states_{P,A}\times\states_{P,B}.
    \end{equation}
    If desired, this can be converted into a scalar quantity given a price $p$ for $B$ in terms of $A$ on an external market:
    \begin{equation} \label{payoff-1d}
      -a + p(P(\phi_A) - P(\phi_A+a)).
    \end{equation}
    However, we will be interested in sequences of swaps which are \emph{pure profit} in that both parameters of the 2d payoff are non-negative.
    
  \item
    Infinitesimally, the reciprocal derivative 
    \[
      -1/P_\Lambda' = \frac{df/dr_A}{df/dr_B}
    \]     
    of the cost potential gives the \emph{marginal pricing} on $C$ of $B$ in terms of $A$.
  
\end{itemize}

\begin{example}
\label{cpmm-pricing}

  In the familiar CPMM (constant product) situation \cite{adams2020uniswap}, we have
  \begin{align*}
    f(x,y) &= xy \\
    P_\Lambda(r) &= \Lambda/r 
  \end{align*}
  from which we obtain the marginal pricing function $-P_\Lambda'(r) = \Lambda/r^2$. If $(x,y) = (r, P_\Lambda(r))$ we recover the more familiar form $xy/x^2=x/y$ for marginal price.

\end{example}

An \emph{instance} of the CFMM $f$ on a state machine $\varrho:\msg\rightarrow\End(\mathcal{X})$ consists of an address key $C\in\addresses$ and a pair of tokens $A,B\in\mathcal{T}$ corresponding to the labelling of the coordinate axes of the domain of $f$.
The relevant part of the local state has the form
\[
  \mathcal{X}_C = \mathcal{X}_{A,C}\times\mathcal{X}_{B,C} \cong (0,\infty)^2.
\]

\subsection{Order flow}

\emph{Orders} define elements of $\msg$ that act on $\states$ through trades that \emph{fulfil} the order.
In this paper we will mostly consider market orders, so that the client only specifies the amount on one side of the swap.
The other side is quoted by the CFMM, which always fulfils the order provided the client has enough funds to cover their side.

We adopt the convention that the amount is always specified on the $A$ side (the num\'eraire).
Thus the space $\msg_C$ of possible orders is $(0,\infty)_\mathrm{BUY}\sqcup(0,\infty)_\mathrm{SELL}$.
Here BUY means ``buy $B$ with $A$,'' so a fulfilled BUY order results in $A$ tokens being transferred from the originator to the CFMM (and $B$ tokens being transferred the other way). 
We identify $(0,\infty)_\mathrm{BUY}\sqcup(0,\infty)_\mathrm{SELL}\simeq \R\setminus\{0\}$ with BUY orders on the positive side.

In order to model the state process on our CFMM, we need some hypothesis about the order flow process $(X_i)_{i\in\N}$ (an $\msg_C = \R\setminus\{0\}$-valued process).
One such hypothesis is the `zero intelligence' hypothesis popular in the econophysics literature \cite{gode1993allocative, chakraborti2011econophysics}:

\begin{definition}[Zero intelligence]
\label{zero-intelligence}

  We say that a message arrivals process $(M_n\stackrel{\$}{\in}\msg)_{n\in\N}$ is \emph{zero intelligence} (ZI) if the $M_i$ are sampled i.i.d.~from a common distribution on $\msg$.

\end{definition}

The ZI hypothesis captures the situation where messages are submitted by independent agents without knowledge of the contents or number of messages already submitted, including by themselves.
Each message $M_i$ is a market order on the market $C$ with some fixed probability $p_C$, and this predicate is a sequence of i.i.d.~Bernoulli trials.
If arrivals of all messages (including injected messages) are Poisson with rate $\lambda$, then arrivals of orders on $C$ is Poisson with rate $p_C\lambda$.
Similarly, under the priority model of \S\ref{priority-model}, high priority orders arrive at a rate $p_C\lambda_\mathrm{high}$ (assuming priority is independent of whether or not a given message is an order on $C$).

The size and direction of the market orders $X_0,\ldots$ on $C$ are sampled from a distribution $F_o$ on $\msg_C=\R\setminus\{0\}$.
Again, we suppose the $X_i$ are independent of priority, so we can treat the high priority queue as its own ZI order flow.

Let us also add a couple of other simplifying assumptions:
\begin{itemize}

  \item
    Consider only orders generated by senders who have sufficient balance to cover their trade at time of execution.
    In the case of SELL orders (buying the num\'eraire), the required balance depends on the state of the CFMM at execution time.
    So these orders are really only `zero-intelligence' if the total volume is negligible compared to reserves.
  
  \item
    No liquidity is added or removed from $C$, and the level sets of $f$ do not hit the boundary of $[0,\infty)^2$. Then $C$'s reserves cannot run out.

\end{itemize}

Under these assumptions, given state $\phi_{C,0}\in\states_{C,0} = (0,\infty)^2$ at time of sending an order $o$, the num\'eraire balance of $C$ at the time of execution of $o$ is
\begin{equation}
  \Phi_{C,A,K} = \varrho(M_0)\cdot\cdots\cdot \varrho(M_K)\cdot \phi_{C,A,0} = \phi_{C,A,0} + \sum_{i=0}^{K} X_i;
\end{equation}
the position of a random walk with step size distribution $F_o$ after $K$ steps.

The moments of the num\'eraire reserves can be quite easy to calculate given sufficent information about the input distributions.
For example, if $X$ is a random variable with distribution $F_o$, we have
\begin{equation}
  \mathbb{E}(\Phi_{C,A,K}) = \phi_{C,A,0} + \mathbb{E}(K)\cdot\mathbb{E}(X)
\end{equation}
\begin{equation}
  \sigma^2(\Phi_{C,A,K}) = \sigma^2(K)\mathbb{E}(X)^2 + \mathbb{E}(K)\sigma^2(X)
\end{equation}
by Wald's identity and the law of total variance.

However, due to the form of the cost potential or pricing function (e.g.~as in Example \ref{cpmm-pricing}), it may be difficult to derive analytic expressions for the total reserves or actual execution price of $o$.
In this case, Monte Carlo methods can be used to obtain estimates of the moments.

\begin{example}[Exponential block time, uniform size, CFMM]

  Let us adopt the priority model with exponential block times as in \eqref{exponential-block-time}, so that $\widetilde K$ is geometrically distributed with failure probability $p$ (a number that is quite easily computed in practice from $\lambda_\mathrm{high}$ and $\mu = 1/\mathbb{E}(T)$).
  Let $X$ be a random variable with distribution $F_o$.
  Then we obtain
  \begin{equation}
    \mathbb{E}(\Phi_{C,A,K}) = \phi_{A,C,0} + \frac{p}{1-p} \cdot \mathbb{E}(X)
  \end{equation}
  \begin{equation}
    \sigma^2(\Phi_{C,A,K}) = \frac{p}{(1-p)^2} \cdot \mathbb{E}(X)^2 + \frac{p}{1-p}\sigma^2(X)
  \end{equation}
  To compute $\sigma^2(\Phi_{C,A,K}
  )$, we need to make an ansatz on the mean and variance of $F_o$.
  Under typical martingale assumptions, $\mathbb{E}(X)=0$, so the expectation becomes simply $\phi_{A,C,0}$ and the variance is $p\sigma^2(X)/(1-p)^2$.
  
  In the zero-intelligence literature it is quite common to use uniformly random direction and size with the latter constrained to some interval $[0,L]$ \cite[\nopp VII.B.3]{chakraborti2011econophysics}.
  The following figures were sampled (50 samples, $\beta=200$) using this model against a simple CPMM:
  \begin{center}
    \begin{tabular}{ c cc | c }
    $\phi_0$ & $L$ & $\lambda/\mu$ & $c_v($price$)\cdot 100\%$\\
    \hline
    $(100,100)$&  $1$& $30$ & \input{assets/cv_100_100_1_30} \\ 
    \hline
    $(120,120)$ & $0.1$ & $14$ & \input{assets/cv_120_120_0.1_12} \\  
    \hline
    $(50,150)$ & $1$ & $4$ & \input{assets/cv_50_150_1_4} \\
    \hline
    $(25,35)$ & $1$ & $1$ & \input{assets/cv_35_25_1_1}
    \end{tabular}
  \end{center}
  (In the present context, $\lambda/\mu$ is the average number of trades executed on the pool $C$ per block and $c_v=\sigma/\mathbb{E}$ is the coefficient of variation.)
  Together with concentration inequalities, these kinds of figures can be used to provide probabilistic guarantees on the discrepancy between the forecast price, i.e.~price at time of order creation, and execution price. For example, Chebyshev gives us
  \begin{equation}
    \Prob\left(\;\left|\frac{\text{executed price}}{\text{forecast price}}-1\right|\geq 2c_v\;\right) \leq 1/4. 
  \end{equation}
  
\end{example}  

Of course, one could just as well apply these methods with real order flow data (extracted, for example, from an index of \texttt{Swap} events on a Uniswapv2 pool).
The model for $K$ means that we do not have to also collect mempool data to obtain concrete figures.

\subsection{Sandwich trades}
\label{sandwich}

\begin{definition}[Sandwich]

  A \emph{sandwich trade} is a an ordered triple of trades $(\tau_-,\tau_0,\tau_+)$ such that:
  \begin{enumerate}
    \item All three trades have the same counterparty $C$ (the \emph{market});
    \item The outer trades $\tau_\pm$ have the same originator $S$ (the \emph{sandwicher}).
  \end{enumerate}
  The originator of the middle trade is called the \emph{victim}.
  
  The sandwich trade is \emph{feasible} on an initialised AMM $(f,C,\phi_0)$ if it fulfils all three trades in sequence, that is, if
  \begin{enumerate}
    \item $(f,\phi_0)$ fulfils $\tau_-$;
    \item $(f,\phi_0 + \tau_-)$ fulfils $\tau_0$;
    \item $(f,\phi_0 + \tau_- + \tau_0)$ fulfils $\tau_+$.
  \end{enumerate}
\end{definition}

Suppose $\tau_0>0$ (i.e.~it is a BUY order).
If $P$ is convex and the sandwicher $S$ has an $A$-balance of at least $\tau_->0$, the sandwich executes and they get a payout 
\[
  (-(\tau_+\tau_+) ,\, P(\phi_0)-P(\phi_0+\tau_-) + P(\phi_0+ \tau_- + \tau_0) - P(\phi_0+ \tau_- + \tau_0 + \tau_+)).
\]
The sandwich is \emph{pure profit} if both co-ordinates are non-negative.
If we assume $\tau_-=-\tau_+=\epsilon$, this simplifies to
\begin{equation}
\label{sandwich-payoff}
  (0, P(\phi_0)-P(\phi_0+\epsilon) + P(\phi_0+ \epsilon + \tau_0) - P(\phi_0+ \tau_0 ))
\end{equation}
and the pure profit condition simplifies to non-negativity of the second co-ordinate.
Note that in this case, the CFMM ends up in the same state as if $\tau_0$ had not been sandwiched.\footnote{Of course, this observation breaks down in the presence of trading fees.}
We call such a sandwich \emph{memoryless}.
Allowing $\epsilon\rightarrow 0_+$, we find that the marginal memoryless sandwich payoff for a target trade of size $a$ is 
\begin{equation}
\label{sandwich-marginal}
  (0,P'(x+a) - P'(x)).
\end{equation}

\begin{proposition}
  The following conditions on a (not necessarily quasi-concave) CFMM $f:(0,\infty)^2\rightarrow\R$ are equivalent:
  \begin{enumerate}
    \item $f$ is (strictly) quasi-concave at $W_\Lambda$.
    \item $P_\Lambda$ is (strictly) convex.
    \item For any $a>0$ the function $P_\Lambda(x+a)-P_\Lambda(x)$ is (strictly) monotone increasing.
  \end{enumerate}
  In this case, any market order may be strictly pure profitably sandwiched at any CFMM state.
\end{proposition}
\begin{proof}

  The equivalence of 1 and 2 is clear as $W_\Lambda$ is the graph of $P$. For $3\Rightarrow 2$, divide through by $a$ and let $a\rightarrow 0_+$ to find that $P'$ is (strictly) increasing, that is $P''>0$. The converse implication is a general property of convex functions.
  
  For the final statement suppose we have a trade $\tau(a)$ executing at some location in a block. The marginal profit from inserting the memoryless sandwich with size $\epsilon$ at exactly that position is given by equation \eqref{sandwich-marginal}; in particular, it is positive for any $x$. \qedhere
  
\end{proof}

\begin{theorem}
\label{sandwich-existence}

  Let $\widetilde\msg\subseteq\msg$ be a set of messages that includes all swaps on $C$ with originator $S$. 
  Let $L>0$, and suppose $S$ has an $A$-balance of at least $L$.
  Let $U:\states_{S,\{A,B\}}=(0,\infty)^2\rightarrow\R$ be a strictly monotone increasing utility function.
  Let $o\in \widetilde\msg$ be a market order with originator $P\neq S$ to buy $A$ in terms of $B$, and let $B\stackrel{\$}{\in}\blocks\left(\widetilde\msg\right)$ be a random block satisfying the conditions:
  \begin{enumerate}
    \item $B$ is weakly rational w.r.t.~$U$.
    \item $\ell(B)\leq \beta$.
  \end{enumerate}
  %
  %Then the following event almost never occurs: $\ell(B)\leq \beta-2$, $o\in B$, and the volume of injected swaps is strictly less than $L$.
  Then given $o\in B$, the probability of at least one of the following events occurring is unity:
  \begin{enumerate}
    \item $\ell(B)\geq \beta-1$;
    \item the volume of injected swaps is at least $L$;
    \item $o$ is sandwiched.
  \end{enumerate}
  
\end{theorem}
\begin{proof}

  Suppose $B\phi_0=(x,y)$.
  If $o\in B$ and the volume $V$ of injected swaps is strictly less than $L$, then the memoryless sandwich of $o$ with sandwicher $S$ and size $\epsilon = L-V$:
  \begin{itemize}
    \item would execute at any position in $B$ (i.e.~$S$ has sufficient balance);
    \item would yield a pure profit in $B$ tokens (and $0$ $A$ tokens), and hence, by monotonicity of $U$, a strictly greater utility than not sandwiching.
  \end{itemize}
  Hence by weak rationality, the probability of $o$ not being sandwiched vanishes.
  \qedhere
  
\end{proof}

With exponential block times as in Example \ref{exponential-block-time}, the event $\ell(B_0)\geq\beta-1$ has probability $1-p^{\beta-1}$.
This quantity is small for all but the highest arrivals rates --- for example, with $\beta=200$ and $\lambda/\mu\beta=0.99$ it is about $0.02$.
So in most cases, we are left with either cases 2 or 3.
If the scheduler operator is well-capitalised, so that $L\gg0$, case 2 also becomes unlikely.
Hence when traffic is not maxed out and the adversary is very well captalised, the probability that $o$ gets sandwiched when it is included into a block is high.

\begin{remark}[How realistic are our utility assumptions?] \label{exotic-derivative}

  Corollary \ref{sandwich-existence} holds under the assumption that the sandwicher's utility delta depends only on the keys $(A,P)$ and $(B,P)$.
  If we `tokenise' all utility deltas, this can be phrased as saying that $S$ does not hold any asset other than $A$ and $B$ whose value depends on the keys $(A,P),(B,P),(A,C),(B,C)$.
  This assumption can be weakened in a few ways:
  \begin{itemize}
    \item 
      Any asset in the portfolio of $S$ has value which is monotone increasing in the balances $\Phi_{A,S},\Phi_{B,S}$.
      (The idea of a derivative of the \emph{balance} a particular address has of a token is quite exotic, but it certainly could exist in principle.)
    \item 
      $S$ may hold assets with dependence on $\Phi_M$ through the \emph{current price}.
  \end{itemize}
  Realistic assets that might violate these hypotheses include volatility derivatives which could in principle depend on prices at various points within the block.
  
  Another one could be some kind of social insurance credit which is invalidated above a certain income or wealth level.
  
\end{remark}

\section{Conclusion}

We have described the outline of a model of a blockchain mempool to which methods from both game or decision theory and queueing theory can be applied, and discussed two simplified limits at a level of detail that yields actual predictions.
With this type of model, protocol designers at the dapp or blockchain level can make predictions about the performance of their applications and hence deliver more accurate guarantees to users.

While these limits by themselves are unrealistic in some ways, it is easy to see where we can begin improving them: more refined approaches to arrivals with some interdependence between orders, incomplete information environments for schedulers, imposition of more interesting verifiable sequencing rules, more careful treatment of special positions such as top of block, and so on.
The real test of our approach will be in applying it to less familiar queueing environments such as those discussed in \S\S\ref{other-work},\ref{future-directions}.

\nocite{angeris2019analysis}
\nocite{angeris2020improved}

\printbibliography

\end{document}

%% file: assets/cv_100_100_1_30
$5.14\%$

%% file: assets/cv_120_120_0.1_12
$0.31\%$

%% file: assets/cv_50_150_1_4
$5.30\%$

%% file: assets/cv_35_25_1_1
$4.56\%$